\documentclass[11pt]{article}

\usepackage{amssymb,amsmath,amsfonts,amsthm}
\usepackage{latexsym}
\usepackage{graphics}
\usepackage{tikz}
\usetikzlibrary{shapes,arrows}
\usepackage{indentfirst}
\usepackage{hyperref}
\usepackage{mathtools}
\usepackage{comment}
\usepackage{bm}

\hypersetup{colorlinks = true, linkcolor = blue, citecolor = blue, urlcolor = blue}

\allowdisplaybreaks

\newtheorem{innercustomthm}{Theorem}
\newenvironment{customthm}[1]
  {\renewcommand\theinnercustomthm{#1}\innercustomthm}
  {\endinnercustomthm}
\newtheorem{innercustompro}{Proposition}
\newenvironment{custompro}[1]
  {\renewcommand\theinnercustompro{#1}\innercustompro}
  {\endinnercustompro}

\newtheorem{innercustomcor}{Corollary}
\newenvironment{customcor}[1]
  {\renewcommand\theinnercustomcor{#1}\innercustomcor}
  {\endinnercustomcor}

\newtheorem*{thm*}{Theorem}
\newtheorem{thm}{Theorem}
\newtheorem{lem}[thm]{Lemma}
\newtheorem{pro}[thm]{Proposition}

\newtheorem{cor}[thm]{Corollary}
\newtheorem{conj}[thm]{Conjecture}

\newtheorem{ques}[thm]{Question}

\newcommand{\N}{\mathbb{N}}

\begin{document}

\title{Shameful Inequalities for List and DP Coloring of Graphs}

\author{Hemanshu Kaul\footnote{Department of Applied Mathematics, Illinois Institute of Technology, Chicago, IL 60616. E-mail: {\tt kaul@iit.edu}},
Jeffrey A. Mudrock\footnote{Department of Mathematics and Statistics, University of South Alabama, Mobile, AL 36688. E-mail: {\tt mudrock@southalabama.edu}},
Gunjan Sharma\footnote{Department of Applied Mathematics, Illinois Institute of Technology, Chicago, IL 60616. E-mail: {\tt gsharma7@hawk.iit.edu}}
}



\maketitle

\begin{abstract}

 The chromatic polynomial of a graph is an important notion in algebraic combinatorics that was introduced by Birkhoff in 1912; denoted $P(G,k)$, it equals the number of proper $k$-colorings of graph $G$. Enumerative analogues of the chromatic polynomial of a graph have been introduced for two well-studied generalizations of ordinary coloring, namely, list colorings:  $P_{\ell}$, the list color function (1990); and DP colorings: $P_{DP}$, the DP color function (2019), and $P^*_{DP}$, the dual DP color function (2021). For any graph $G$ and $k \in \N$, $P_{DP}(G, k) \leq P_\ell(G,k) \leq P(G,k) \leq P_{DP}^*(G,k)$.  In 2000, Dong settled a conjecture of Bartels and Welsh from 1995 known as the \emph{Shameful Conjecture} by proving that for any $n$-vertex graph $G$, $P(G,k+1)/(k+1)^n \geq P(G,k)/k^n$ for all $k \in \N$ satisfying $k \geq n-1$.  In contrast, for infinitely many positive integers $n$, Seymour (1997) gave an example of an $n$-vertex graph for which the above inequality does not hold for some $k = \Theta(n/ \log n)$. In this paper, we consider analogues of Dong's result for list and DP color functions.  Specifically, in contrast to the chromatic polynomial, we prove that for any $n$-vertex graph $G$, $P_{\ell}(G,k+1)/(k+1)^n \geq P_{\ell}(G,k)/k^n$ and $P_{DP}(G,k+1)/(k+1)^n \geq P_{DP}(G,k)/k^n$  for all $k \in \N$.  For the dual DP analogue of these inequalities, we show that there is a graph $G$ and $k \in \N$ such that $P_{DP}^*(G,k+1)/(k+1)^n < P_{DP}^*(G,k)/k^n$, and we prove $P_{DP}^*(G,k+1)/(k+1)^n \geq P_{DP}^*(G,k)/k^n$ for all $k \in \N$ satisfying $k \geq n-1$ when $G$ is an $n$-vertex complete bipartite graph.
\medskip



\end{abstract}

\section{Introduction}\label{intro}

We begin by establishing some notation and basic definitions.  In this paper all graphs are nonempty, finite, and simple unless otherwise noted.  Generally speaking we follow West~\cite{W01} for terminology and notation. We use $\mathbb{N}$ to denote the set of all positive integers. For $m \in \mathbb{N}$, $[m]$ denotes the set $\{1,\ldots,m\}$. For a set $S$, $2^{S}$ denotes the power set of $S$. For a graph $G$, $V(G)$ and $E(G)$ denote the vertex set and the edge set of $G$ respectively.  When $u$ and $v$ are adjacent in $G$ we use $uv$ or $vu$ to denote the edge with endpoints $u$ and $v$.  By $H \subseteq G$, we mean that $H$ is a subgraph of $G$. If $S \subseteq V(G)$, $G[S]$ denotes the subgraph of $G$ induced by $S$. For any $v \in V(G)$, we let $G-v$ be the subgraph $G[V(G)-\{v\}]$. For any $S \subseteq E(G)$, we let $G-S$ be the subgraph of $G$ with vertex set $V(G)$ and edge set $E(G)-S$. For any $S_{1},S_{2} \subseteq V(G)$, $E_{G}(S_{1},S_{2})$ denotes the set consisting of the edges in $E(G)$ that have one endpoint in $S_{1}$ and the other in $S_{2}$. We use $K_{m,n}$ to denote the complete bipartite graphs with partite sets of size $m$ and $n$. The \emph{join} of vertex disjoint graphs $G$ and $H$ is the graph obtained from their disjoint union by adding the edges in $\{uv : u \in V(G), v \in V(H)\}$.

\subsection{Graph coloring, list coloring and DP-coloring}

In the classic vertex coloring problem, we seek a \emph{proper $k$-coloring} of a graph $G$ which is a function $f: V(G) \rightarrow [k]$ such that $f(v) \neq f(u)$ whenever $uv \in E(G)$. We say that $G$ is \emph{$k$-colorable} if it has a proper $k$-coloring. The \emph{chromatic number} $\chi(G)$ of $G$ is the smallest $k \in \mathbb{N}$ such that there exists a proper $k$-coloring of $G$.  
 
List coloring is a generalization of classic vertex coloring that was introduced in the 1970s independently by Vizing~\cite{V76} and Erd\H{o}s, Rubin, and Taylor~\cite{ET79}. A \emph{list assignment} of $G$ is a function $L$ that assigns a set of colors to each $v \in V(G)$. If $|L(v)| = k$ for each $v \in V(G)$, then $L$ is called a \emph{$k$-assignment} of $G$. The graph $G$ is \emph{$L$-colorable} if there exists a proper coloring $f$ of $G$ such that $f(v) \in L(v)$ for each $v \in V(G)$ (we refer to $f$ as a \emph{proper $L$-coloring} of $G$). The \emph{list chromatic number} of $G$, $\chi_{\ell}(G)$, is the smallest $k$ such that there exists a proper $L$-coloring for every $k$-assignment $L$ of $G$. It immediately follows that for any graph $G$, $\chi(G) \leq \chi_\ell(G)$. This inequality may be strict since it is known that the gap between $\chi(G)$ and $\chi_{\ell}(G)$ can be arbitrarily large; for example, $\chi_{\ell}(K_{m,n}) = m+1$ when $n \ge m^m$, but all bipartite graphs are $2$-colorable.

 DP-coloring is a generalization of list coloring that was introduced by Dvo\v{r}\'{a}k and Postle~\cite{DP15} in 2015. Intuitively, DP-coloring is a generalization of list coloring where each vertex in the graph still gets a list of colors but identification of which colors are different can vary from edge to edge. Formally, a \emph{DP-cover} of a graph $G$ is a pair $\mathcal{H} = (L,H)$, where $H$ is a graph and $L : V(G) \rightarrow 2^{V(H)}$ is a function such that:

\vspace{3mm}

     \noindent(1) the set $\{L(v) : v \in V(G)\}$ forms a partition of $V(H)$ of size $|V(G)|$;\\
     (2) for every $v \in V(G)$, the graph $H[L(v)]$ is a complete graph;\\
     (3) if $E_{H}(L(u),L(v))$ is nonempty, then either $u = v$ or $uv \in E(G)$;\\
     (4) if $uv \in E(G)$, then $E_{H}(L(u),L(v))$ is a matching (the matching may be empty).
     
\vspace{3mm}     
 
In this paper, a cover always refers to a DP-cover. Suppose $\mathcal{H} = (L,H)$ is a cover of $G$. We refer to the edges of $H$ connecting distinct parts of the partition $\{L(v) : v \in V(G) \}$ as \emph{cross-edges}. A cover $\mathcal{H} = (L,H)$ of $G$ is \emph{$k$-fold} if $|L(v)| = k$ for each $v \in V(G)$. An $\mathcal{H}$ -coloring of $G$ is an independent set in $H$ of size $|V(G)|$.  The \emph{DP-chromatic number} of $G$, $\chi_{DP}(G)$, is the smallest $k \in \mathbb{N}$ such that $G$ has an $\mathcal{H}$-coloring for every $k$-fold cover $\mathcal{H}$ of $G$. 

A $k$-fold cover $\mathcal{H} = (L,H)$ of a graph $G$ is called \emph{full} if for each $uv \in E(G)$, $|E_{H}(L(u),L(v))| = k$. Given a $k$-fold cover $\mathcal{H}$ of $G$, unless otherwise mentioned, we will always name the vertices of $H$ so that for each $v \in V(G)$, $L(v) = \{(v,i) : i \in [k]\}$. 

Suppose $\mathcal{H} = (L,H)$ is a $k$-fold cover of $G$.  We say that $\mathcal{H}$ has a \emph{canonical labeling} if it is possible to name the vertices of $H$ so that $L(u) = \{ (u,j) : j \in [k] \}$ and $(u,j)(v,j) \in E(H)$ for each $j \in [k]$ whenever $uv \in E(G)$.\footnote{When $\mathcal{H}=(L,H)$ has a canonical labeling, we will always refer to the vertices of $H$ using this naming scheme.} An easy inductive argument can be used to show that a full cover of a tree has a canonical labeling (see~\cite{KM20}). It is easy to show that if $\mathcal{H}$ is a $k$-fold cover of graph $G$ such that $\mathcal{H}$ has a canonical labeling, then $G$ has a proper $k$-coloring if and only if $G$ admits an $\mathcal{H}$-coloring.
 
Given a cover $\mathcal{H} = (L,H)$ of $G$ and subgraph $G'$ of $G$, the \emph{subcover of $\mathcal{H}$ restricted to $G'$} is $\mathcal{H}'=(L',H')$ where $L'$ is the restriction of $L$ to $V(G')$ and $H'$ is the subgraph of $H$ with $V(H')=\bigcup_{u\in V(G')}L(u)$ and $E(H') = E(H[V(H')]) - \bigcup_{uv \in E(G)-E(G')} E_H(L(u),L(v))$.
 
 Given a $k$-assignment $L$ for a graph $G$, it is easy to construct a $k$-fold cover $\mathcal{H}$ of $G$ such that $G$ has an $\mathcal{H}$-coloring if and only if $G$ has a proper $L$-coloring (see~\cite{BK17}).  It follows that $\chi_\ell(G) \leq \chi_{DP}(G)$.  This inequality may be strict since it is easy to prove that $\chi_{DP}(C_n) = 3$ whenever $n \geq 3$, but the list chromatic number of any even cycle is $2$ (see~\cite{BK17} and~\cite{ET79}).

 \subsection{Counting Colorings} \label{DPCF}
 
For $k \in \N$, the \emph{chromatic polynomial} of a graph $G$, $P(G,k)$, is equal to the number of proper $k$-colorings of $G$. It is known that $P(G,k)$ is a polynomial in $k$ of degree $|V(G)|$ (see~\cite{B12}).

In 1990 (\cite{AS90}), the notion of chromatic polynomial was extended to list coloring as follows. If $L$ is a list assignment for $G$, we use $P(G,L)$ to denote the number of proper $L$-colorings of $G$. The \emph{list color function} $P_\ell(G,k)$ is the minimum value of $P(G,L)$ where the minimum is taken over all possible $k$-assignments $L$ for $G$.  Since a $k$-assignment could assign the same $k$ colors to every vertex in a graph, it is clear that $P_\ell(G,k) \leq P(G,k)$ for each $k \in \N$. It is known that for each $k \in \N$, $P(G,k)=P_{\ell}(G,k)$ when $G$ is a cycle or chordal\footnote{A chordal graph is a graph in which all cycles of length four or more contain a chord.} (see~\cite{KN16} and~\cite{AS90}). However, for some graphs, the list color function can differ significantly from the chromatic polynomial for small values of $k$.  One reason for this is that a graph can have a list chromatic number that is much higher than its chromatic number. On the other hand, Dong and Zhang~\cite{DZ22} (improving upon results in~\cite{D92},~\cite{T09} and~\cite{WQ17}) recently showed that for a connected graph $G$ with $t \geq 2$ edges, $P_{\ell}(G,k)=P(G,k)$ whenever $k \geq t-1$. 
     
The notion of chromatic polynomial has also been extended to DP-coloring~\cite{KM20}.  Suppose $\mathcal{H} = (L,H)$ is a cover of graph $G$.  Let $P_{DP}(G, \mathcal{H})$ be the number of $\mathcal{H}$-colorings of $G$.  Then, the \emph{DP color function of $G$}, denoted $P_{DP}(G,k)$, is the minimum value of $P_{DP}(G, \mathcal{H})$ where the minimum is taken over all possible $k$-fold covers $\mathcal{H}$ of $G$.  It is easy to see that for any graph $G$ and $k \in \N$, $P_{DP}(G, k) \leq P_\ell(G,k) \leq P(G,k).$
  
It is fairly straightforward to prove that for each $n \geq 3$ and $k \in \N$, $P_{DP}(C_n,k) = P(C_n,k) = (k-1)^n - (k-1)$ when $n$ is odd. However for $n \geq 4$ and $k \geq 2$, $P_{DP}(C_n,k) = (k-1)^n - 1 < (k-1)^n + (k-1) = P(C_n,k)$ when $n$ is even (see~\cite{KM20} and~\cite{M21}). So, unlike list color functions, DP color functions need not be equal to the corresponding chromatic polynomial for sufficiently large $k$.  
  
Furthermore, Dong and Yang~\cite{DY21} (extending results of~\cite{KM20}) showed that if $G$ contains an edge $e$ such that the length of a shortest cycle containing $e$ in $G$ is even, then there exists an $N \in \mathbb{N}$ such that $P_{DP}(G,k) < P(G,k)$ whenever $k \geq N$. See~\cite{ZD22} for a further generalization of this. Asymptotically, it was shown in~\cite{MT20} that for every $n$-vertex graph $G$, $P(G, k)-P_{DP}(G, k) = O(k^{n-3})$ as $k \rightarrow \infty$.

Recently, the second named author~\cite{M21} introduced the \emph{dual DP color function} of a graph $G$, denoted $P^{*}_{DP}(G,k)$, which equals the maximum number of $\mathcal{H}$-colorings of $G$ over all full $k$-fold covers $\mathcal{H}$ of $G$. Suppose $\mathcal{H}$ has a canonical labeling and $G$ has a proper $k$-coloring.  Then, if $\mathcal{I}$ is the set of $\mathcal{H}$-colorings of $G$ and $\mathcal{C}$ is the set of proper $k$-colorings of $G$, the function $f: \mathcal{C} \rightarrow \mathcal{I}$ given by $f(c) = \{ (v, c(v)) : v \in V(G) \}$ is a bijection. In fact, 
$$P_{DP}(G,k) \leq P_{\ell}(G,k) \leq P(G,k) \leq P^{*}_{DP}(G,k).$$
It is known that for each $n \geq 3$, if $n$ is even and $k \in \N$, $P_{DP}^{*}(C_{n}, k) = P(C_{n}, k) = (k-1)^n +(k-1)$.  Also if $n$ is odd and $k \geq 2$, $P_{DP}^{*}(C_{n}, k) = (k-1)^n + 1 > P(C_n,k)$ (see~\cite{M21}). So, like DP color functions, dual DP color functions need not be equal to the corresponding chromatic polynomial for sufficiently large $k$.  It is also known that (see~\cite{M21}) for every $n$-vertex graph $G$, $P^{*}_{DP}(G, k) - P(G, k) = O(k^{n-2})$ as $k \rightarrow \infty$. 

\subsection{The Shameful Conjecture and our Motivation} \label{section:SC}

Let $n \geq 2$ and for an $n$-vertex graph $G$, let $\mu(G)$ be the expected number of colors in a uniformly random proper $n$-coloring of $G$. In 1995, to approximate $\mu(G)$, Bartels and Welsh~\cite{BW95} studied a Markov chain on colorings and showed that $\mu(G) = n\left(1 - P(G,n-1)/P(G,n)\right).$ It is easy to see that over all graphs with $n$ vertices, $\mu(G)$ is maximized when $G$ is a complete graph. Even though it seems intuitively obvious that as we keep adding edges to a graph, the expected number of colors needed to color that graph can only increase, it is not true in general. On the other hand, it is seems reasonable that over all $n$-vertex graphs, $\mu(G)$ must be minimized when $G$ is an empty graph on $n$ vertices, $O_{n}$. Bartels and Welsh conjectured that for any $n$-vertex graph $G$, $\mu(G) \geq \mu(O_{n})$. Since $\mu(O_{n}) = n(1 - {((n-1)}/{n})^{n})$, they conjectured,
\begin{equation*}
        n\left(1 - \frac{P(G,n-1)}{P(G,n)}\right) \geq n\left(1 - \left(\frac{n-1}{n}\right)^{n}\right),
    \end{equation*}

\noindent or equivalently,
\begin{equation*}
        \frac{P(G,n)}{n^n} \geq \frac{P(G,n-1)}{(n-1)^n}.
    \end{equation*}

Bartels and Welsh could not prove this at the time even though they thought it was intuitively obvious. Hence they called it the \emph{Shameful Conjecture}. In 1997, Seymour~\cite{S97}, making progress towards this Conjecture, showed that ${P(G,n)}/{P(G,n-1)} \geq {685}/{252}.$ Notice that $(n/(n-1))^n > {685}/{252}$.  The Shameful Conjecture was ultimately proved by Dong~\cite{D00} in 2000. In fact, he proved a stronger result and showed the following. 
  
\begin{thm} [\cite{D00}]   \label{thm: Dong}
            For any $n$-vertex graph $G$, $\frac{P(G,k+1)}{(k+1)^n} \geq \frac{P(G,k)}{k^n}$  for all $k \in \N$ satisfying $k \geq n-1$.
\end{thm} 

The lower bound on $k$ in Theorem~\ref{thm: Dong} has been studied in terms of the maximum degree of a graph as well. In 2015, Fadnavis~\cite{F15} showed that ${P(G,k+1)}/{(k+1)^{n}} \geq {P(G,k)}/{k^{n}}$ whenever $k \geq 36\Delta^{3/2}$, where $\Delta$ is the maximum degree of $G$.  

Henceforth, for a graph $G$ and $k \in \N$, we will refer to the inequality ${P(G,k+1)}/{(k+1)^{n}} \geq {P(G,k)}/{k^{n}}$ as \emph{the Shameful $(G,k)$-Inequality}\footnote{Whenever it is clear from the context, we will simply say the Shameful Inequality instead of the Shameful $(G,k)$-Inequality}. Note that the Shameful Inequality may not be true for all $k \geq 2$. This can be shown by the following example given by McDiarmid (see~\cite{F15}). Consider $G = K_{n,n}$ where $n \geq 10$. Then, ${P(G,3)}/{P(G,2)} = {(6 + 6(2^n-2))}/{2} = {3(2^n-1)}/{1} < {3^{2n}}/{2^{2n}}$. Seymour~\cite{S97} showed that for any $q \in \N$ if $G$ is a complete $(2^q)$-partite graph with partite sets of size $100q$, then ${P(G,2^q+1)}/{(2^q+1)^n} < {P(G,2^q)}/{2^{qn}}$ and $|V(G)|/(1000\log(|V(G)|)) \leq 2^q \leq |V(G)|/ \log(|V(G)|)$. Interestingly, these are the only counterexamples for smaller values of $k$ of which we are aware.
   
Now we state a well known probabilistic description of the Shameful $(G,k)$-Inequality. This viewpoint will be useful in understanding our results. Suppose each vertex in $G$ is colored independently and uniformly at random with a color from $[k]$. Since there are a total of $k^n$ possible $k$-colorings of $G$, ${P(G,k)}/{k^n}$ is the probability that a randomly chosen $k$-coloring of $G$ is a proper $k$-coloring of $G$. Let $c_{k}$ denote a uniformly, randomly chosen $k$-coloring of $G$. Then the $(G,k)$-Shameful Inequality can be interpreted as:
\begin{equation*}
        \mathbb{P}[c_{k+1} \text{~is proper}] \geq \mathbb{P}[c_{k} \text{~is proper}].
    \end{equation*}
Intuitively, it may seem this inequality should always be true. But as we have just seen, it does not hold for all graphs and values of $k$. Motivated by this, we are interested in studying the list color function, the DP color function, and the dual DP color function analogues of the Shameful $(G,k)$-Inequality. The following question is the main focus of this paper.

\begin{ques} \label{ques:main}
Given an $n$-vertex graph $G$ and $k \in \N$, do the following inequalities hold?\\ 

(i)  $\frac{P_{\ell}(G,k+1)}{(k+1)^n} \geq \frac{P_{\ell}(G,k)}{k^n}$.\\
        
(ii)  $ \frac{P_{DP}(G,k+1)}{(k+1)^n} \geq \frac{P_{DP}(G,k)}{k^n}$.\\

(iii)  $ \frac{P^{*}_{DP}(G,k+1)}{(k+1)^n} \geq \frac{P^{*}_{DP}(G,k)}{k^n}$.
\end{ques}

We refer to (i), (ii) and (iii) as \emph{the list Shameful $(G,k)$-Inequality}, \emph{the DP Shameful $(G,k)$-Inequality} and \emph{the Dual-DP Shameful $(G,k)$-Inequality} respectively. Moreover, we will omit saying $(G,k)$ if it is clear from the context.  

\subsection{Outline and Discussion of Results}

Recall that for a given graph, for large enough values of $k$, its list color function equals its chromatic polynomial. On the other hand, for certain graphs, the DP color function need not be equal to the corresponding chromatic polynomial, no matter how large $k$ gets. Moreover, we know the Shameful $(G,k)$-Inequality does not hold for all graphs $G$ and $k \in \N$. However, despite the aforementioned difference in behavior between the list color function and the DP color function with respect to the corresponsing chromatic polynomial, we show that the list Shameful $(G,k)$-Inequality and the DP Shameful $(G,k)$-Inequality hold for all graphs $G$ and $k \in \N$. We begin Section~\ref{sec:shameful} by proving the following.   

\begin{thm} \label{thm: listSC}
For any $n$-vertex graph $G$, $\frac{P_{\ell}(G,k+1)}{(k+1)^n} \geq \frac{P_{\ell}(G,k)}{k^n}$ for all $k \in \N$.
\end{thm}

A graph $G$ is called \emph{enumeratively chromatic-choosable} if $P_{\ell}(G, k) = P(G, k)$ for all $k \geq \chi(G)$ (see~\cite{AM25,KMUG23}). Chordal graphs~\cite{AS90} and cycles~\cite{KN16} are known examples of enumeratively chromatic-choosable graphs. It is also known that if $G$ is enumeratively chromatic-choosable, then $G \vee K_{n}$ is enumeratively chromatic-choosable for every $n \in \N$ (see~\cite{KM21}). See more examples and discussion in~\cite{AM25}. Theorem~\ref{thm: listSC} immediately implies that enumeratively chromatic-choosable graphs satisfy the Shameful $(G,k)$-inequality for all $k \in \N$.

\begin{cor} \label{cor: enumcc}
    If $G$ is enumeratively chromatic-choosable, then $\frac{P(G,k+1)}{(k+1)^n} \geq \frac{P(G,k)}{k^n}$ for all $k \in \N$.
\end{cor}

As a straightforward application of Theorem~\ref{thm: listSC}, if the Shameful $(G,k)$-Inequality does not hold then $P_{\ell}(G,k) < P(G,k)$. 

\begin{cor}  \label{cor: cool}
 Given an $n$-vertex graph graph $G$ and $k \in \N$, let $c_{k}$ denote a uniformly chosen random $k$-coloring of $G$. If $\mathbb{P}[c_{k+1} \text{~is proper}] < \mathbb{P}[c_{k} \text{~is proper}]$, then $P_{\ell}(G,k) < P(G,k)$.
\end{cor}

Recall that when $G = K_{n,n}$ and $n \geq 10$, ${P(G,3)}/{3^{2n}} < {P(G,2)}/{2^{2n}}$.  So, as an illustration of Corollary~\ref{cor: cool} we can conclude that $P_{\ell}(K_{n,n},2) < P(K_{n,n},2)$ for all $n \geq 10$ which can also be proven directly.  In Theorem~\ref{thm: DPSC}, we prove that the DP Shameful $(G,k)$-Inequality holds for all graphs $G$ and $k \in \N$.

\begin{thm} \label{thm: DPSC}
For any $n$-vertex graph $G$, $\frac{P_{DP}(G,k+1)}{(k+1)^n} \geq \frac{P_{DP}(G,k)}{k^n}$ for all $k \in \N$.
\end{thm}

In Section~\ref{sec: Dual} we consider the dual DP color function. Since the dual DP color function can be explicitly calculated for cycles (see~\cite{M21}) and $K_{4}$ (see~\cite{M24}), it is easy to show that the Dual-DP Shameful $(G,k)$-Inequality holds for all values of $k$ when $G$ is a cycle or a $K_{4}$. However, the Dual-DP Shameful $(G,k)$-inequality does not hold for all graphs $G$ and $k \in \N$.

\begin{pro} \label{pro:sharpk5k6}
   If $G$ is $K_{n}$ where $n \geq 5$ or $K_{n,n}$ where $n \geq 10$, then the Dual-DP Shameful $(G,2)$-Inequality does not hold.  
\end{pro}

Note that finding an explicit formula for $P^{*}_{DP}(K_n,k)$ is an open problem and is not known from $n=5$ onwards. This makes the proof of the proposition above, as well as answering the next natural question, more challenging than might be expected. 

\begin{ques} \label{ques: dual clique}
For each $n$, what is the smallest $k \geq 2$ such that the Dual-DP Shameful $(K_n,k)$-Inequality holds?
\end{ques}

As discussed above, the Dual-DP Shameful $(K_n,k)$-Inequality holds for all $k$ when $n=2,3,4$, while for $n \geq 5$ we see that this is not true. However, we expect that for every $n$-vertex graph $G$ the Dual-DP Shameful $(G,k)$-Inequality is true for all $k \in \N$ satisfying $k \ge n-1$.

\begin{conj}\label{conj: dual}
   For every $n$-vertex graph $G$,  $ \frac{P^{*}_{DP}(G,k+1)}{(k+1)^n} \geq \frac{P^{*}_{DP}(G,k)}{k^n}$ for all $k \in \N$ satisfying $k \ge n-1$.
\end{conj}

As some evidence towards this conjecture we prove that it holds for all complete bipartite graphs. This is a direct consequence of the following result.

\begin{thm} \label{thm:dual}
Let $m,n,k \in \N$. Then $P^{*}_{DP}(K_{m,n},k) = P(K_{m,n},k)$.   
\end{thm}   

Interestingly, Theorem~\ref{thm:dual} is proven with the help of the Rearrangement Inequality without directly working with the formulae for the chromatic polynomials of complete bipartite graphs. Applying Theorems~\ref{thm: Dong} and~\ref{thm:dual} we get the following result. 

\begin{cor} \label{cor: application}
    Let $m, n, k \in \N$. Then $\frac{P^{*}_{DP}(K_{m,n},k+1)}{(k+1)^{m+n}} \geq \frac{P^{*}_{DP}(K_{m,n},k)}{k^{m+n}}$ for all $k \geq m+n-1$.
\end{cor}

\section{List and DP Shameful Inequalities} \label{sec:shameful}

In this section, we answer Question~\ref{ques:main}(i) and Question~\ref{ques:main}(ii) in Theorem~\ref{thm: listSC} and Theorem~\ref{thm: DPSC} respectively.

\begin{customthm} {\bf \ref{thm: listSC}} 
For any $n$-vertex graph $G$, $\frac{P_{\ell}(G,k+1)}{(k+1)^n} \geq \frac{P_{\ell}(G,k)}{k^n}$ for all $k \in \N$.
\end{customthm}

\begin{proof}

Let $L$ be a $(k+1)$-assignment of $G$ such that $P(G,L) = P_{\ell}(G,k+1)$, and let $a = P_{\ell}(G,k+1)$. Let $\{l_{i} : i \in [a]\}$ be the set of proper $L$-colorings of $G$. Let $S$ be the set of $k$-assignments of $G$ obtained by deleting one element from $L(v)$ for each $v \in V(G)$. Let $T = \{P(G,L') : L' \in S\}$. For each $v \in V(G)$, we delete one element from $L(v)$ uniformly at random. Let $Y$ denote the resulting $k$-assignment. Clearly $Y$ is an element of $S$.

For each $i \in [a]$, suppose $E_{i}$ is the event that $l_{i}$ is a proper $Y$-coloring of $G$. Note that $l_{i}$ is a proper $Y$-coloring of $G$ if and only if $l_{i}(v) \in Y(v)$ for each $v \in V(G)$. Since for each $v \in V(G)$ the probability that an element from $L(v)$ is deleted is $1/(k+1)$, $\mathbb{P}[E_{i}] = (1-1/(k+1))^{n} = k^{n}/(k+1)^n$. Let $X_{i}$ be the indicator random variable such that $X_{i} = 1$ when $E_{i}$ occurs and $X_{i} = 0$ otherwise. Let $X = \Sigma_{i \in [a]} X_{i}$. Equivalently, $X$ equals the number of proper $Y$-colorings of $G$ in $\{l_{i} : i \in [a]\}$. Since there is no proper $Y$-coloring of $G$ that is not an element of the set $\{l_{i} : i \in [a]\}$, we have $X = P(G,Y)$. By linearity of expectation, $\mathbb{E}[X] = ak^{n}/(k+1)^n = P_{\ell}(G,k+1)k^{n}/(k+1)^n$. Hence there exists an element, $t \in T$ such that $t \leq \mathbb{E}[X]$. Thus there exists an $L^{*}\in S$ such that $P_{\ell}(G,L^{*}) \leq P_{\ell}(G,k+1)k^{n}/(k+1)^n$, and since $P_{\ell}(G,k) \leq P_{\ell}(G,L^{*})$, the result follows.
\end{proof}

The following result is a straightforward application of Theorem~\ref{thm: listSC}.

\begin{customcor}  {\bf \ref{cor: cool}}
    Given a graph $G$ and $k \in \N$, let $c_{k}$ denote a uniformly chosen random $k$-coloring of $G$. If $\mathbb{P}[c_{k+1} \text{~is proper}] < \mathbb{P}[c_{k} \text{~is proper}]$, then $P_{\ell}(G,k) < P(G,k)$.
\end{customcor}

\begin{proof}
It suffices to show that $P_{\ell}(G,k) \neq P(G,k)$. For the sake of contradiction, suppose $P_{\ell}(G,k) = P(G,k)$. Then we have ${P_{\ell}(G,k+1)}/{(k+1)^n} \leq {P(G,k+1)}/{(k+1)^n} < {P(G,k)}/{k^n} = {P_{\ell}(G,k)}/{k^n}$ which contradicts Theorem~\ref{thm: listSC}.    
\end{proof}

Now we extend the proof of Theorem~\ref{thm: listSC} to the context of DP-coloring.

\begin{customthm} {\bf \ref{thm: DPSC}}  
For any $n$-vertex graph $G$, $\frac{P_{DP}(G,k+1)}{(k+1)^n} \geq \frac{P_{DP}(G,k)}{k^n}$ for all $k \in \N$.
\end{customthm}

\begin{proof}

Let $\mathcal{H} = (L,H)$ be a $(k+1)$-fold cover of $G$ such that $P_{DP}(G,\mathcal{H}) = P_{DP}(G,k+1)$, and let $a = P_{DP}(G,k+1)$. Let $\{I_{i} : i \in [a]\}$ be the set of $\mathcal{H}$-colorings of $G$. Let $S$ be the set of $k$-fold covers of $G$ obtained by deleting one element from $L(v)$ for each $v \in V(G)$. Let $T = \{P_{DP}(G,\mathcal{H}') : \mathcal{H}' \in S\}$. For each $v \in V(G)$, we delete one element from $L(v)$ uniformly at random. Let $\mathcal{Y}$ denote the resulting $k$-fold cover of $G$. Suppose $\mathcal{Y} = (L_{\mathcal{Y}}, H_{\mathcal{Y}})$. Clearly $\mathcal{Y}$ is an element of $S$.

For each $i \in [a]$, suppose $E_{i}$ is the event that $I_{i}$ is a $\mathcal{Y}$-coloring of $G$. Note that $I_{i}$ is a $\mathcal{Y}$-coloring of $G$ if and only if $I_{i}\cap L(v) \in L_{\mathcal{Y}}(v)$ for each $v \in V(G)$. Since for each $v \in V(G)$ the probability that an element from $L(v)$ is deleted is $1/(k+1)$, $\mathbb{P}[E_{i}] = (1-1/(k+1))^{n} = k^{n}/(k+1)^n$. Let $X_{i}$ be the indicator random variable such that $X_{i} = 1$ when $E_{i}$ occurs and $X_{i} = 0$ otherwise. Let $X = \Sigma_{i \in [a]} X_{i}$. Equivalently, $X$ equals the number of $\mathcal{Y}$-colorings of $G$ in $\{I_{i} : i \in [a]\}$. Since there are no $\mathcal{Y}$-colorings of $G$ outside the set $\{I_{i} : i \in [a]\}$, we have $X = P_{DP}(G,\mathcal{Y})$. By linearity of expectation, $\mathbb{E}[X] = ak^{n}/(k+1)^n = P_{DP}(G,k+1)k^{n}/(k+1)^n$. Hence there exists an element, $t \in T$ such that $t \leq \mathbb{E}[X]$. Thus there exists an element, $\mathcal{H}^{*}$ in $S$ such that $P_{DP}(G,\mathcal{H}^{*}) \leq P_{DP}(G,k+1)k^{n}/(k+1)^n$, and since $P_{DP}(G,k) \leq P_{DP}(G,\mathcal{H}^{*})$, the result follows.
\end{proof}

\section{Dual-DP Shameful Inequality}\label{sec: Dual}

The dual DP color function can be explicitly calculated for cycles (see~\cite{M21}) and $K_{4}$ (see~\cite{M24}). For each $n \geq 3$, if $n$ is even and $k \in \N$, $P_{DP}^{*}(C_{n}, k) = P(C_{n}, k) = (k-1)^n +(k-1)$.  Also if $n$ is odd and $k \geq 2$, $P_{DP}^{*}(C_{n}, k) = (k-1)^n + 1 > P(C_n,k)$. Furthermore, for even values of $k \geq 2$, $P_{DP}^{*}(K_{4}, k) = k^4-6k^3+15k^2-13k$ and for odd $k \geq 3$, $P_{DP}^{*}(K_{4}, k) = k^4-6k^3+15k^2-13k-3$. It is easy to show that the Dual-DP Shameful $(G,k)$-Inequality holds for all $k \in \N$ when $G$ is a cycle or a $K_{4}$.

In this section we will first focus on proving that Dual-DP Shameful $(G,k)$-inequality does not hold for all graphs $G$ and $k \in \N$ (Proposition~\ref{pro:sharpk5k6}). Since finding an explicit formula for $P^{*}_{DP}(K_n,k)$ is an open problem and is not known from $n=5$ onwards, the proof of this Proposition is more challenging than might be expected. We will need three results before proving Proposition~\ref{pro:sharpk5k6}.

\begin{pro} \label{pro: duallower}
    Let $G$ be a graph and $k \geq 2$. Then $P_{DP}^{*}(G, k) \geq 2$.
\end{pro}

\begin{proof}
Notice that if $G$ is edgeless then $P_{DP}^{*}(G, k) = k^{|V(G)|} \geq 2$. Suppose $|E(G)| \geq 1$. Note that if $G$ is a disconnected graph with components: $G_1, \ldots, G_t$, then $P^*_{DP}(G, k) = \prod_{i=1}^{t} P^*_{DP}(G_i, k)$. So we may assume that $G$ is connected. Since $P_{DP}^{*}(G, k) \geq P_{DP}^{*}(G, 2)$ for $k \geq 2$, it suffices to construct a full $2$-fold cover $\mathcal{H} = (L, H)$ of $G$ such that $P_{DP}(G,\mathcal{H}) \geq 2$.  Let $L(v) = \{(v,i): i \in [2]\}$ for each $v \in V(G)$.  Now, we construct $H$ as follows.  First, let $V(H) = \bigcup_{v \in V(G)} L(v)$, and create edges so that $(v,1)(v,2) \in E(H)$ for each $v \in V(G)$.  Finally, for each $uv \in E(G)$, create edges so that $(u,1)(v,2)$, $(u,2)(v,1) \in E(H)$. Clearly $\{(v,1):v \in V(G)\}$ and $\{(v,2):v \in V(G)\}$ are $\mathcal{H}$-colorings of $G$. Thus $P_{DP}(G,\mathcal{H}) \geq 2$, and the desired result follows.
\end{proof}

\begin{pro} [\cite{DKM23}, \cite{KM20}] \label{pro:canonical} 
Suppose $G$ is a graph and $\mathcal{H} = (L,H)$ is a full $k$-fold cover of $G$. Suppose also that $F$ is an acyclic subgraph of $G$.  Then, the subcover of $\mathcal{H}$ restricted to $F$ has a canonical labeling.
\end{pro}

Recall that a vertex $u$ in $G$ is a \emph{universal vertex} if $uv \in E(G)$ for each $v \in V(G)-\{u\}$.

\begin{lem} \label{pro:dualthreeupper}
    Suppose $G$ is a graph with universal vertex $w$ such that $|V(G)| \geq 2$. Let $G' = G - w$. Then $P_{DP}^{*}(G, 3) \leq P_{DP}^{*}(G', 3) + 3$. Consequently for each $n \in \N$, $P_{DP}^{*}(K_{n}, 3) \leq 3n$.
\end{lem}

\begin{proof}
    Suppose $\mathcal{H} = (L,H)$ is a 3-fold cover of $G$ such that $P_{DP}(G, \mathcal{H}) = P_{DP}^{*}(G, 3)$.  By Proposition~\ref{pro:canonical} we may assume that for each $u \in V(G)$, $L(u) = \{ (u,i) : i \in [3] \}$, and for each $i \in [3]$ and $v \in V(G')$, $(v,i)(w,i) \in E(H)$. Let $\mathcal{H}' = (L', H')$ be the subcover of $\mathcal{H}$ restricted to $G'$.  Suppose $\mathcal{A}$ is the set of $\mathcal{H}$-colorings of $G$ and suppose $\mathcal{B}$ is the set of $\mathcal{H}'$-colorings of $G'$. Let $f: \mathcal{A} \rightarrow \mathcal{B}$ be the function given by $f(I) = I - (I \cap L(w))$. Let $\mathcal{B}_{1}$ be the subset of $\mathcal{B}$ such that $B \in \mathcal{B}_1$ if and only if there is an $x \in [3]$ such that every element of $B$ has second coordinate $x$. Let $\mathcal{B}_{2}$ be the subset of $\mathcal{B}$ such that $B \in \mathcal{B}_2$ if and only if $B \notin \mathcal{B}_1$ and there is an $x \in [3]$ such that every element of $B$ has second coordinate in the set $[3] - \{x\}$. Finally, let $\mathcal{B}_{3} = \mathcal{B} - (\mathcal{B}_1 \cup \mathcal{B}_2)$. Notice that $\{\mathcal{B}_{1}, \mathcal{B}_{2}, \mathcal{B}_{3}\}$ is a partition of $\mathcal{B}$ and $|\mathcal{B}_{1}| \leq 3$. Clearly, 
    $$P_{DP}(G, \mathcal{H}) = |\mathcal{A}| = \sum_{I' \in \mathcal{B}} |f^{-1}(I')| = \sum_{I' \in \mathcal{B}_{1}} |f^{-1}(I')| + \sum_{I' \in \mathcal{B}_{2}} |f^{-1}(I')| + \sum_{I' \in \mathcal{B}_{3}} |f^{-1}(I')|.$$ 
    Note that each $\mathcal{H}'$-coloring of $G'$ in $\mathcal{B}_{1}$ is a subset of exactly two $\mathcal{H}$-colorings of $G$.  Moreover, each $\mathcal{H}'$-coloring of $G'$ in $\mathcal{B}_{2}$ is a subset of exactly one $\mathcal{H}$-coloring of $G$.  Finally notice that no $\mathcal{H}'$-colorings of $G'$ in $\mathcal{B}_{3}$ are subsets of an $\mathcal{H}$-coloring of $G$. This means that 
    \begin{align*}
    P_{DP}(G, \mathcal{H}) = \sum_{I' \in \mathcal{B}_{1}} |f^{-1}(I')| + \sum_{I' \in \mathcal{B}_{2}} |f^{-1}(I')| + \sum_{I' \in \mathcal{B}_{3}} |f^{-1}(I')| &= 2 |\mathcal{B}_{1}| + |\mathcal{B}_{2}| + 0 \\ 
    &\leq 3 + |\mathcal{B}_{1}| + |\mathcal{B}_{2}| \\
    &\leq 3 + |\mathcal{B}| \\ 
    &\leq 3 + P_{DP}^{*}(G',3).    
    \end{align*}
Finally, the fact that $P^{*}_{DP}(K_n,3) \leq 3n$ follows from the fact that $P^{*}_{DP}(K_1,3)=3$ and an easy inductive argument.
\end{proof}

We are now ready to prove Proposition~\ref{pro:sharpk5k6}.

\begin{custompro} {\bf \ref{pro:sharpk5k6}}
    If $G$ is a $K_{n}$ where $n \geq 5$ or $K_{n,n}$ where $n \geq 10$, then the Dual-DP Shameful $(G,2)$-Inequality does not hold.  
\end{custompro}

\begin{proof}
If $G$ is a $K_{n}$ with $n \geq 5$ we need to show that ${P^{*}_{DP}(G,3)}/{3^n} < {P^{*}_{DP}(G,2)}/{2^n}$. From Lemma~\ref{pro:dualthreeupper} it follows that $P_{DP}^{*}(G, 3) \leq 3n$. Furthermore, by Proposition~\ref{pro: duallower}, $P_{DP}^{*}(G, 2) \geq 2$. Thus $P^{*}_{DP}(G,3)/{3^n} \leq 3n/3^{n}$ and $P^{*}_{DP}(G,2)/{2^n} \geq 2/2^{n}$. Since $3n/3^{n} < 2/2^{n}$ whenever $n \geq 5$, we have $P^{*}_{DP}(G,3)/{3^n} \leq 3n/3^{n} < 2/2^{n} \leq P^{*}_{DP}(G,2)/{2^n}$.

For $G = K_{n,n}$ where $n \geq 10$, the desired result follows from the fact that the Shameful $(G,2)$-Inequality doesn't hold (see Section~\ref{intro}) and Theorem~\ref{thm:dual} below.
\end{proof}

\subsection{Dual-DP Shameful Inequality for Complete Bipartite graphs} \label{sec:diff}

In this section, we make further progress towards Question~\ref{ques:main}(iii) in the case of complete bipartite graphs. In Theorem~\ref{thm:dual}, we show that for complete bipartite graphs, the dual DP color function always equals the chromatic polynomial. We prove this without directly working with the formulae for the chromatic polynomial or the dual DP color function of complete bipartite graphs. We first state a version of the Rearrangement Inequality which we will use in the proof of Theorem~\ref{thm:dual}. 

\begin{thm} [Rearrangement Inequality, \cite{R52}] \label{thm:rearrangement}
    Let $k, t \in \N$. For each $(i,j) \in [k] \times [t]$, suppose $y_{ij} \geq 0$ and $y_{i1} \leq \ldots \leq y_{it}$ for each $i \in [k]$. For each $i \in [k]$, let $\sigma_{i}$ be an arbitrary permutation of $[t]$. Then, $\sum_{j=1}^{t} \prod_{i=1}^{k} y_{i\sigma_{i}(j)} \leq \sum_{j=1}^{t} \prod_{i=1}^{k} y_{ij}$.   
\end{thm}

We are now ready to prove Theorem~\ref{thm:dual}.

\begin{customthm} {\bf \ref{thm:dual}} 
Let $m,n,k \in \N$. Then $P^{*}_{DP}(K_{m,n},k) = P(K_{m,n},k)$.   
\end{customthm}

\begin{proof}
The statement is obvious when $k = 1$; so, henceforth we will assume that $k \geq 2$. Let $K$ be the complete bipartite graph with bipartition $X = \{x_{1}, \ldots, x_{m} \}$ and $Y = \{y_{1}, \ldots, y_{n}\}$. Let $\mathcal{H} = (L, H)$ be an arbitrary full $k$-fold cover of $K$ and let $\mathcal{H}^{*} = (L, H^{*})$ be a $k$-fold cover of $K$ with a canonical labeling so that for each $p \in [m]$ and $q \in [n]$, $(x_{p},l)(y_{q},l) \in E(H^{*})$ for each $l \in [k]$. It suffices to show that $P_{DP}(K,\mathcal{H}) \leq P_{DP}(K,\mathcal{H}^{*})$. Note that there are $k^{n}$ ways to pick an independent set of size $n$ from $H^{*}[\bigcup_{y \in Y} L(y)]$ and from $H[\bigcup_{y \in Y} L(y)]$. Moreover these independent sets are the same. We index these independent sets as $A_{1}, \ldots, A_{k^{n}}$. 

We start with counting the number of $\mathcal{H}^{*}$-colorings of $K$. For each $i \in [m]$ and $j \in [k^n]$, let $\alpha_{ij}$ be the number of vertices, $(x_{i},z) \in L(x_{i})$ such that $A_{j}\cup \{(x_{i},z)\}$ is an independent set in $H^{*}$. Since $X$ is an independent set, the number of $H^{*}$-colorings of $K$ that have $A_{j}$ as a subset are $\prod_{i=1}^{m} \alpha_{ij}$. Thus $P_{DP}(K,\mathcal{H}^{*}) = \sum_{j=1}^{k^{n}}\prod_{i=1}^{m} \alpha_{ij}$. Notice that for each $i \in [m]$ and $j \in [k^{n}]$, $\alpha_{ij} \in \{0\}\cup[k-1] $. Recall that for each $p \in [m]$ and $q \in [n]$, $(x_{p},l)(y_{q},l) \in E(H^{*})$ for each $l \in [k]$. Hence for each $j \in [k^n]$, $\alpha_{1j} = \ldots = \alpha_{mj}$. Suppose $\tau$ is a permutation of $[k^{n}]$ so that $\alpha_{i\tau(1)} \leq \ldots \leq \alpha_{i\tau(k^{n})}$ for each $i \in [m]$. Then we let $\beta_{i1} = \alpha_{i\tau(1)}, \ldots, \beta_{ik^{n}} = \alpha_{i\tau(k^{n})}$ for each $i \in [m]$. Clearly $\sum_{j=1}^{k^{n}}\prod_{i=1}^{m} \alpha_{ij} = \sum_{j=1}^{k^{n}}\prod_{i=1}^{m} \beta_{ij}$. 

Now we count the number of $\mathcal{H}$-colorings of $K$. For each $i \in [m]$ and $j \in [k^n]$, let $\gamma_{ij}$ be the number of vertices, $(x_{j},z) \in L(x_{j})$ so that $A_{i}\cup \{(x_{j},z)\}$ is an independent set in $H$. For each $i \in [m]$, let $c_{i,l}$ be the number of times $l$ appears in $\beta_{i1}, \ldots, \beta_{ik^{n}}$. Since $\mathcal{H}$ is full, Proposition~\ref{pro:canonical} tells us that for every $p,q \in [m]$, $c_{p,l} = c_{q,l}$. Hence for each $i \in [m]$, the multiset $\{\gamma_{ij} : j \in [k^{n}]\}$ equals the multiset $\{\beta_{ij} : j \in [k^{n}]\}$. Clearly $P_{DP}(K,\mathcal{H}) = \sum_{j=1}^{k^{n}}\prod_{i=1}^{m} \gamma_{ij} = \sum_{j=1}^{k^{n}}\prod_{i=1}^{m} \beta_{i\sigma_{i}(j)}$ where $\sigma_{i}$ is a permutation of $[k^{n}]$ for each $i \in [m]$. By applying Theorem~\ref{thm:rearrangement}, we have $P_{DP}(K, \mathcal{H}) = \sum_{j=1}^{k^{n}}\prod_{i=1}^{m} \beta_{i\sigma_{i}(j)} \leq \sum_{j=1}^{k^{n}}\prod_{i=1}^{m} \beta_{ij} = P_{DP}(K,\mathcal{H}^{*})$.
\end{proof}

Finally, as an application of Theorem~\ref{thm: Dong} we get,

\begin{customcor} {\bf \ref{cor: application}} 
     Let $m, n, k \in \N$. Then, $\frac{P^{*}_{DP}(K_{m,n},k+1)}{(k+1)^{m+n}} \geq \frac{P^{*}_{DP}(K_{m,n},k)}{k^{m+n}}$ for all $k \geq m+n-1$.
\end{customcor}

\end{document}